\pdfoutput=1
\documentclass{siamart0516}

\usepackage{lipsum}
\usepackage{amsfonts}
\usepackage{graphicx}
\usepackage{epstopdf}
\usepackage{algorithmic}

\usepackage{amsmath}

\ifpdf
  \DeclareGraphicsExtensions{.eps,.pdf,.png,.jpg}
\else
  \DeclareGraphicsExtensions{.eps}
\fi

\newcommand{\TheTitle}{A Revisit of Block Power Methods for Finite State Markov Chain Applications} 
\newcommand{\TheAuthors}{Hao Ji, Seth H. Weinberg, and Yaohang Li}

\headers{\TheTitle}{\TheAuthors}

\title{{\TheTitle}
}

\author{
  Hao Ji\thanks{Department of Computer Science, California State Polytechnic University, Pomona, CA (\email{hji@cpp.edu}, \url{http://www.cpp.edu/\string~hji/}).}
  \and
  Seth H. Weinberg\thanks{Department of Biomedical Engineering, Virginia Commonwealth University, Richmond, VA (\email{shweinberg@vcu.edu}, \url{http://biomedical.egr.vcu.edu/faculty/seth-h-weinberg-ph-d/}).}
  \and
   Yaohang Li\thanks{Department of Computer Science, Old Dominion University, Norfolk, VA (\email{yaohang@cs.odu.edu}, \url{http://www.cs.odu.edu/\string~yaohang/}).}
}

\usepackage{amsopn}

\ifpdf
\hypersetup{
  pdftitle={\TheTitle},
  pdfauthor={\TheAuthors}
}
\fi

\begin{document}

\maketitle

\begin{abstract}
According to the fundamental theory of Markov chains, under a simple connectedness condition, iteration of a Markov transition matrix $P$, on any random initial state probability vector, will converge to a unique stationary distribution, whose probability vector corresponds to the left eigenvector associated with the dominant eigenvalue $\lambda_1 = 1$. The corresponding convergence speed is governed by the magnitude of the second dominant eigenvalue $\lambda_2$ of $P$. Due to the simplicity to implement, this approach by using power iterations to approximate the dominant eigenpair, known as the power method, is often applied to estimate the stationary distributions of finite state Markov chains in many applications.

In this paper, we revisit the generalized block power methods for approximating the eigenvector associated with $\lambda_1 = 1$ of a Markov chain transition matrix. Our analysis of the block power method shows that when $s$ linearly independent probability vectors are used as the initial block, the convergence of the block power method to the stationary distribution depends on the magnitude of the $(s+1)$th dominant eigenvalue $\lambda_{s+1}$ of $P$ instead of that of $\lambda_2$ in the power method. Therefore, the block power method with block size $s$ is particularly effective for transition matrices where $|\lambda_{s+1}|$ is well separated from $\lambda_1 = 1$ but $|\lambda_2|$ is not. This approach is particularly useful when visiting the elements of a large transition matrix is the main computational bottleneck over matrix--vector multiplications, where the block power method can effectively reduce the total number of times to pass over the matrix. To further reduce the overall computational cost, we combine the block power method with a sliding window scheme, taking advantage of the subsequent vectors of the latest $s$ iterations to assemble the block matrix. The sliding window scheme correlates vectors in the sliding window to quickly remove the influences from the eigenvalues whose magnitudes are smaller than $|\lambda_{s}|$ to reduce the overall number of matrix--vector multiplications to reach convergence. Finally, we compare the effectiveness of these methods in a Markov chain model representing a stochastic luminal calcium release site.
\end{abstract}



\section{Introduction}
\label{sec:1}

Finding the stationary (equilibrium) distribution of a finite state Markov chain is of great interest in many practical applications \cite{zhang1999finite,tauchen1986finite,wang1995finite}. Let $P$ denote an $n \times n$ probability transition matrix of a finite state Markov chain $K$. Based on the fundamental theorem of Markov chains \cite{feller2008introduction}, if $K$ is irreducible, aperiodic, and positive--recurrent, there is a unique stationary distribution characterized by a probability vector $\pi$ such that 
$$\pi^T P=\pi^T.$$
Here $\pi$ corresponds to the left eigenvector associated with the dominant eigenvalue $\lambda_1=1$  of $P$. It is also well known that the convergence rate of the Markov chain $K$ is geometric, which is controlled by the magnitude of the second dominant eigenvalue $\lambda_2$.

The power method \cite{golub2012matrix} is a simple numerical algorithm often applied to approximate the stationary distribution vector of a finite state Markov chain. Starting from a random probability distribution vector $x_0$, the power method is described by the power iteration,
$$x_{i+1}=P^T x_i.$$
The rationale of the power method is that after $k$ power iterations, the magnitudes of the rest eigenvalues other than $\lambda_1$ in $(P^T)^k$ are close to $0$ if $k$ is sufficiently large and then $P^T x_k$ is a good approximation to $\pi$.  As a result, the convergence rate of the power iteration is governed by the magnitude of the second dominant eigenvalue $\lambda_2$ of $P^T$. Despite its slow convergence when $|\lambda_2|$ is close to $1$, the power method has been popularly used in a variety of important Markov chain applications, particularly when the transition matrix $P$ is very large and sparse. For example, the power method is used in Google's PageRank algorithm to rank webpages in the Internet in their search engine results \cite{page1997pagerank}; Twitter employs the power method to recommend ``who to follow'' to its users \cite{kim2014twilite}; by exploring the graph constructed via content and link features, the power method is also applied to calculate the trust vector as the stationary distribution vector of the graph to fight spams \cite{yang2008fighting}; and computing the dominant eigenvectors of the random walk transition matrices of large--scale biological networks using power iterations can reveal important relationships among gene expressions, protein functions, and drugs \cite{lan2016predicting,zhao2016new}.

The block power method is a generalization of the power method by extending the power iterations to the block power iterations (a.k.a. subspace iterations, orthogonal iterations, or simultaneous iterations) \cite{mccormick1977simultaneous,stewart1976simultaneous,rutishauser1969computational} on a block matrix consisting of multiple columns. The block power method is often used in applications where multiple eigenvalues/eigenvectors are of interest. Using block form can also lead to convergence rate enhancement in numerical linear algebra algorithms such as finding solutions of linear systems \cite{hy16bfbcg} and solving least square problems \cite{hy16bcgls}.  In this paper, we revisit the generalized block power method and theoretically analyze its convergence when only the eigenvector associated with the dominant eigenvalue is of interest, as in many Markov chain applications. Then, we introduce a sliding window scheme, combined with the block power method, to reduce the overall number of matrix-vector multiplications to reach convergence by assembling a block matrix with a sliding window from the prior iterations. Correlating vectors in the sliding window is expected to remove the influences from the eigenvalues whose magnitudes are smaller than $|\lambda_{s}|$ and thus leads to a faster approximation to the dominant eigenvector. Finally, the numerical effectiveness of the block power method and the sliding window scheme is demonstrated in a Markov chain application of modeling a stochastic luminal calcium release site.

\section{Methods}
\label{sec:2}

\subsection{The block power method}
\label{subsec:2.1}

The block power method is a generalization of the power method. The generalized block power method consists of an iteration step and a decomposition step. In the iteration step, block power iteration is employed to compute the multi--dimensional invariant subspace. For a non--Hermitian matrix $P$, let $X_0$ be an $n \times s$ matrix with orthonormal columns and the following subspace iteration process generates a series of matrices $X_i$.\\
\begin{tabbing}
\hspace*{1.5cm}\= \hspace*{0.5cm}\= \hspace*{3.5cm}\= \kill 
	\>\textbf{for} $i = 1, 2, \cdots, k,\cdots$ \\
	\>\>$Z_i=P^T X_{i-1}$ \\
	\>\>$X_i R_i=Z_i$ \>/* QR factorization of $Z_i$ */  \\
	\>\textbf{end for} \\
\end{tabbing}
where $s$ is referred to as the block size of the block power method. As a result, $X_i$ tends towards the invariant subspace of $P$ with respect to the eigenvalues $\lambda_1,\lambda_2, \cdots,\lambda_s$, sorted by their magnitudes. Generally, the block power method is used to estimate the top $s$ eigenvalues with largest magnitudes and their associated eigenvectors of a matrix, which converges at a rate proportional to $|\lambda_{s+1} |/|\lambda_s |$ \cite{golub2012matrix}. 

The block power iteration can be tailored for Markov chain applications. Since the dominant eigenvalue $\lambda_1$ in the transition matrix $P$ is known to be $1$, the normalization step in the subspace iteration process is no longer necessary and the block power method can be simplified as\\
\begin{tabbing}
\hspace*{1.5cm}\= \hspace*{0.5cm} \= \kill 
	\>\textbf{for} $i = 1, 2, \cdots, k,\cdots$ \\
	\> \>$X_i=P^T X_{i-1}$ \\
	\>\textbf{end for} \\
\end{tabbing}

Due to the fact that only the eigenvector associated with the dominant eigenvalue is of interest in Markov chain applications, the approximate eigenvector can be extracted from the block matrix $X_k$ by the following Schur--Rayleigh--Ritz step \cite{stewart1976simultaneous} in the decomposition step:\\

\begin{tabbing}
\hspace*{1.5cm}\= \hspace*{4cm} \= \kill 
    \>$Q_k R_k=X_k$  \>/* QR Decomposition */ \\
    \>$B_k=Q_k^T P^T Q_k$  \>/* Projection */ \\
    \>$U_k T_k U_k^T=B_k$  \>/* Schur Decomposition */ \\
    \>$Y_k=Q_k U_k.$ \\
\end{tabbing}

Provided that in the Schur decomposition of $B_k$, $U_k$ is chosen so that the diagonal elements of $T_k$ appear in non--increasing order by their absolute values, $Y_k^{(1)}$, the first column vector in $Y_k$, is an approximation of the dominant left eigenvector $v_1$ of $P$. The decomposition step involves numerical linear algebra operations including matrix--matrix multiplications, QR decomposition, and Schur decomposition, which is significantly more computationally costly than that of an iteration step. Fortunately, the decomposition step only needs to be carried out every certain (large) number of iteration steps when evaluating convergence and extracting the approximated dominant eigenvector are needed.

\subsection{Convergence analysis}
\label{subsec:2.2}
Theorem \ref{theorem1} shows that $Y_k^{(1)}$ converges to the dominant left eigenvector $v_1$ of $P$ at a rate of $O(|\lambda_{s+1}|^k )$. The proof of Theorem \ref{theorem1} relies on the result of Lemma \ref{lemma1} regarding the Schur--Rayleigh--Ritz approximation of an upper triangular matrix whose Schur vectors are the corresponding columns of the identity matrix. Thus, Lemma \ref{lemma1} is a special case of Theorem 3.2 stated in \cite{stewart1976simultaneous}, where the Markov transition matrix is considered and only the dominant eigenvector is of interest.

\begin{lemma}
\label{lemma1}
Suppose that $S$ is an upper triangular matrix where $\lambda_1=1$, $\lambda_2$, $\lambda_3$, $\cdots$, $\lambda_n$ are diagonal elements appearing in non--increasing order by their magnitudes. Let $\widetilde{Y}_k$ denote the Schur--Rayleigh--Ritz approximation corresponding to $S^k X_0$, where $X_0$ is an $n\times s$ initial block matrix and $k$ is the number of subspace iterations. Then, 
$$\emph{dist}\left( span \left\lbrace \widetilde{Y}_k^{(1)} \right\rbrace,span \left\lbrace I^{(1)} \right\rbrace  \right)=O(|\lambda_{s+1}|^k ),$$
where $\emph{dist}(\cdot,\cdot)$ denotes  the distance between these two spaces and $\widetilde{Y}_k^{(1)}$ and $I^{(1)}$ are the first column of matrix $\widetilde{Y}_k$ and identity matrix, respectively.
\end{lemma}

\begin{theorem}
\label{theorem1}
Let $X_0$ be an $n\times s$ initial block matrix and let $Y_k$ be the Schur--Rayleigh--Ritz approximation corresponding to $(P^T)^k X_0$. Assuming that the Schur decomposition of $B_k$ results in an upper triangular matrix $T_k$ with all diagonal elements sorted in non--increasing order by their magnitudes, then vector $Y_k^{(1)}$ converges to the left dominant eigenvector $v_1$ of $P$ at a rate that is $O(|\lambda_{s+1} |^k )$, such that
$$\emph{dist}\left( span \left\lbrace Y_k^{(1)} \right\rbrace,span \left\lbrace v_1 \right\rbrace  \right)=O(|\lambda_{s+1}|^k ).$$
\end{theorem}
\begin{proof}
Suppose that $WSW^{-1}=P^T$ is a Schur decomposition of $P^T$, where $W$ is a unitary matrix consisting of the Schur vectors and the Schur form $S$ is an upper triangular matrix with its diagonal elements as $\lambda_1=1,\lambda_2,\lambda_3,\cdots,\lambda_n$ placed in non--increasing order according to their magnitudes. Then the projection of $(P^T)^k$ onto $X_0$ becomes
\begin{eqnarray*}X_k&=&(P^T)^k X_0=(WSW^{-1})^k X_0 \\
&=& WS^k W^{-1} X_0 \\
&=& WS^k \widetilde{X}_0,
\end{eqnarray*}
where $\widetilde{X}_0$ denotes $W^{-1} X_0$.

Let $Q_k R_k$ be a QR decomposition of $X_k$ and then we can get $W^{-1} Q_k R_k=S^k \widetilde{X}_0$. Because both $W^{-1}$ and $Q_k$ are orthogonal matrices, denoting $\widetilde{Q}_k=W^{-1} Q_k$, we can find that $\widetilde{Q}_k$ is a basis for the range of $S^k \widetilde{X}_0$.

Afterward, considering projecting $P^T$ onto $Q_k$ to construct an $s\times s$ matrix $B_k$, we have
$$B_k=Q_k^T P^T Q_k=\widetilde{Q}_k^T W^{-1} WSW^{-1} W\widetilde{Q}_k=\widetilde{Q}_k^T S\widetilde{Q}_k.$$
Clearly, the Schur vectors $U_k$ of $B_k$ are also the Schur vectors of $\widetilde{Q}_k^T S\widetilde{Q}_k$. Let $\widetilde{Y}_k$ denote the Schur--Rayleigh--Ritz approximation corresponding to $S^k \widetilde{X}_0$. Then, the Schur--Rayleigh--Ritz approximation $Y_k$ to $X_k$ can be expressed as
$$Y_k=Q_k U_k=W\widetilde{Q}_k U_k=W\widetilde{Y}_k.$$
Based on the assumption that all diagonal elements in $T_k$ are sorted in non--increasing order by their magnitudes and according to Lemma \ref{lemma1}, we have 

$$\text{dist}\left( span \left\lbrace \widetilde{Y}_k^{(1)} \right\rbrace,span \left\lbrace I^{(1)} \right\rbrace  \right)=O(|\lambda_{s+1}|^k )$$
and thus
$$\text{dist}\left( span \left\lbrace Y_k^{(1)} \right\rbrace,span \left\lbrace W^{(1)} \right\rbrace  \right)=O(|\lambda_{s+1}|^k ).$$

Since $\lambda_1=1,\lambda_2,\lambda_3,\cdots,\lambda_n$ are placed in non--increasing order by their magnitudes in $S$, $W^{(1)}$ is exactly the dominant left eigenvector $v_1$ of matrix $P$. Hence, we can conclude that $Y_k^{(1)}$ converges to the dominant left eigenvector of matrix $P$ at a rate of $O(|\lambda_{s+1} |^k )$.  \\
\end{proof}

Theorem \ref{theorem1} indicates that in the block power method, when $s$ linearly independent probability vectors are evolved simultaneously under Markov transitions, correlating these vectors can lead to a faster approximation of the dominant eigenvector. The convergence rate, instead of the well--known one related to $|\lambda_2|$ in the fundamental Markov chain theory, now depends on the $(s+1)$th dominant eigenvalue $|\lambda_{s+1} |$ of $P$. An intuitive explanation of the block power method is that the block power iteration is able to quickly remove the influence from the eigenvalues whose magnitudes are smaller than  $|\lambda_{s}|$ and afterwards the Schur--Rayleigh--Ritz step is used as a direct method carried out on the resulted block matrix to rapidly extract the approximated dominant eigenvector. Therefore, the block power method is particularly powerful for Markov chains where $|\lambda_{s+1} |$ and $1$ are well separated but $|\lambda_2 | $ and $1$ are not. 

Theorem \ref{theorem1} assumes that the upper triangular matrix $T_k$ generated in Schur decomposition of $B_k$ has all diagonal elements sorted in non--increasing order by their magnitudes. However, this is not always guaranteed in Schur decomposition in practice. Therefore, instead, eigendecomposition is applied on $B_k$ to obtain the Ritz pairs to approximate the eigenvalue/eigenvector pairs:\\
\begin{tabbing}
\hspace*{1.5cm}\= \hspace*{4cm} \= \kill 
	\>$V_k \Lambda_k V_k^{-1}=B_k$	\>/* Eigendecomposition */\\
	\>$Y_k=Q_k V_k. $  \\
\end{tabbing}
The Ritz vector corresponding to the largest Ritz value is then used as the approximate dominant eigenvector of $P^T$. 

In many large--scale Markov chain applications, the transition matrix $P$ is either stored outside of the system memory or the elements of $P$ need to be regenerated at each use \cite{hy16bfbcg,ji2015monte,ji2016spark}. As a result, passing over the elements of $P$ becomes the main computational bottleneck, which is much more time-consuming than the floating point operations of matrix--vector multiplications and matrix decompositions. In this case, the cost of multiplying a large matrix with a skinny block matrix is not much greater than that of multiplying it with a single vector. The block power method is hence particularly attractive due to its capability of reducing the total number of iterations, i.e., the overall number of times  passing over $P$, as shown in Theorem \ref{theorem1}. Moreover, the multiplication of $P$ with the block matrix can be implemented with high parallel efficiency: in loosely coupled systems these matrix--vector multiplications can be carried out independently while in tightly coupled systems the space locality of the elements in the block matrix can be taken advantage to enhance cache efficiency.

\subsection{A simple illustrative example}
\label{subsec:2.3}

We use the following simple example to illustrate how the block power method can remove the influence from the eigenvalues whose magnitudes are smaller than  $|\lambda_{s}|$. FIG. \ref{fig:fig1} shows the transition diagram of an irreducible Markov chain with five states $\{1,2,3,4,5\}$. The corresponding transition matrix $P$ of the Markov chain shown in FIG. \ref{fig:fig1} is
\[
  P =
  \left[ {\begin{array}{ccccccccc}
   0        ~&~ 0       ~&~ 0      ~&~ 0.9090 ~&~ 0.0910\\
   0.0002   ~&~ 0       ~&~ 0.9998 ~&~ 0      ~&~ 0\\
   0        ~&~ 0.9998  ~&~ 0      ~&~ 0.0002 ~&~ 0\\
   0.6690   ~&~ 0       ~&~ 0      ~&~ 0      ~&~ 0.3310\\
   0.9989   ~&~ 0.0011  ~&~ 0      ~&~ 0      ~&~ 0\\
  \end{array} } \right].
\]

\begin{figure}[!htbp]
\centering
\includegraphics[scale=0.5]{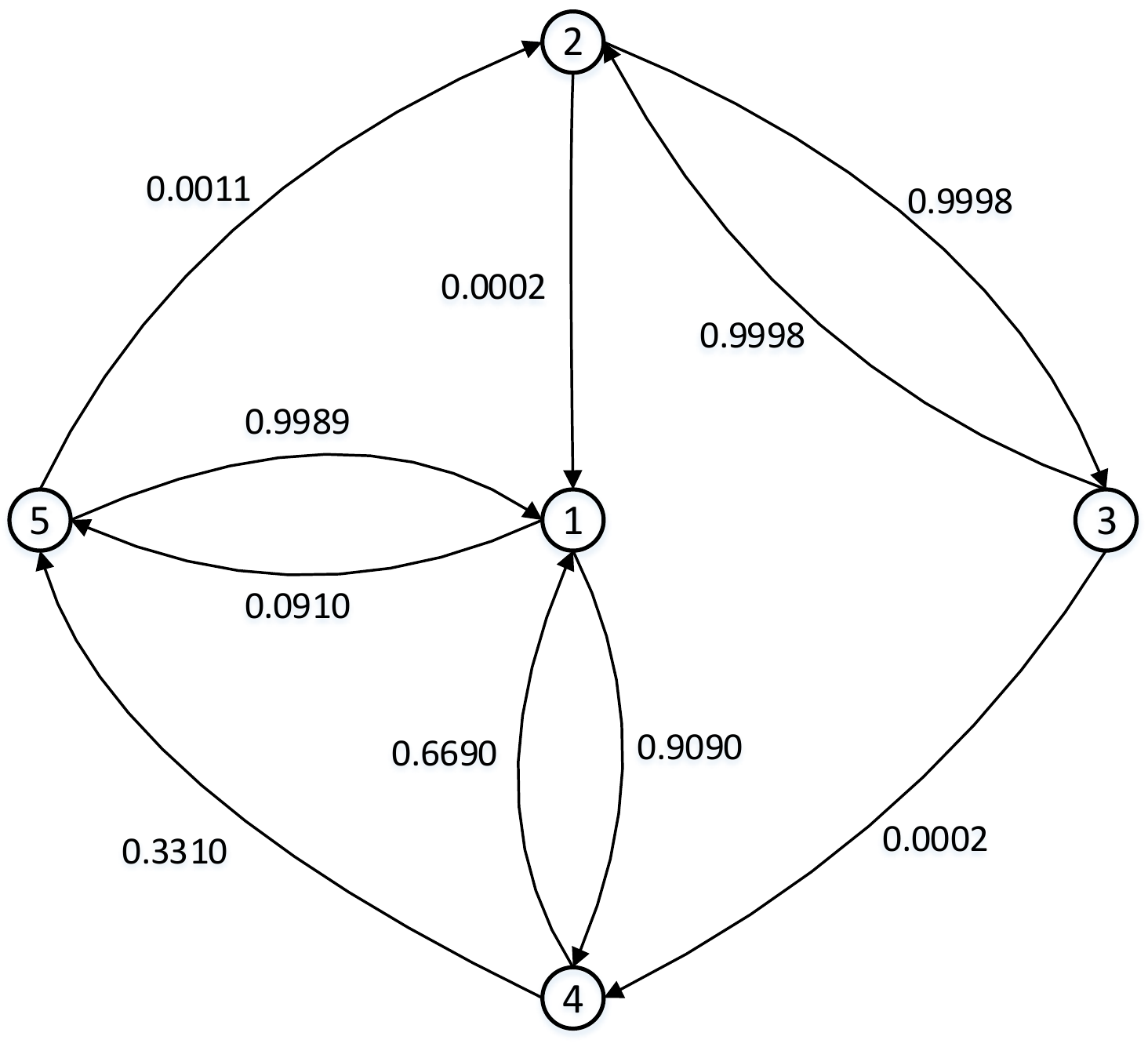}
\caption{The transition diagram of an irreducible Markov chain with five states.}
\label{fig:fig1}       
\end{figure}

\begin{figure}[!htbp]
\centering
\includegraphics[scale=0.25]{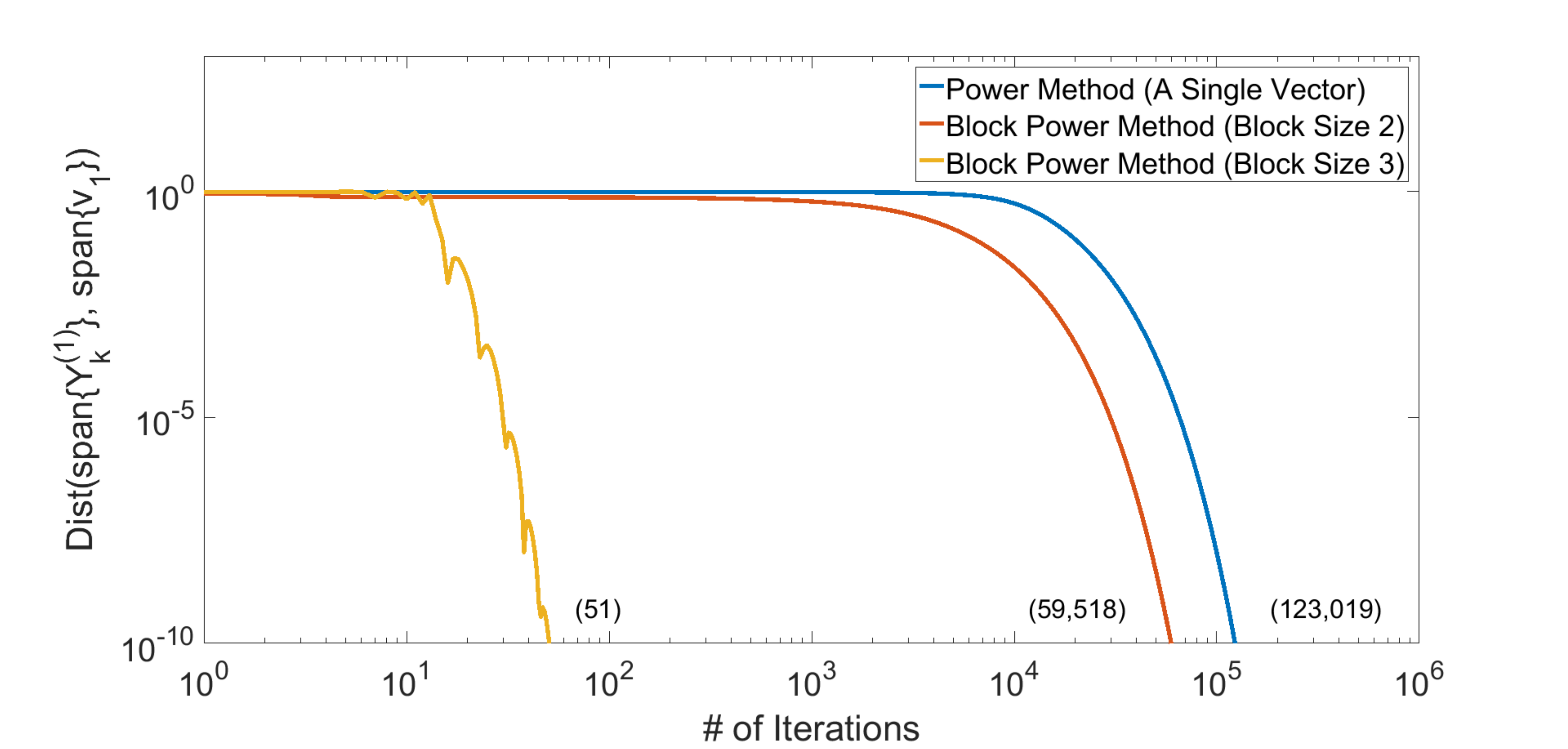}
\caption{Convergence comparison of the power method and the block power method (block size $2$ and $3$) in terms of number of iteration steps for the model shown in FIG. \ref{fig:fig1}.}
\label{fig:fig2}       
\end{figure}

FIG. \ref{fig:fig2} compares the convergence of using the power method as well as the block power method with block sizes $2$ and $3$ to compute the stationary distribution of the Markov chain in this example. In this example, the power method takes $123,019$ iteration steps to converge to a solution with accuracy $10^{-10}$, while the block power method with block size $2$ takes $59,518$ steps and the one with block size $3$ takes only $51$ steps. This is due to the fact that in the transition matrix of this Markov chain example, the gaps between the dominant eigenvalue ($|\lambda_1 |=1$) and the $2$nd and $3$rd dominant eigenvalues ($|\lambda_2 |=0.9998$, $|\lambda_3 |=0.9996$) are small, but the $4$th dominant eigenvalue ($|\lambda_4 |=0.5483$) is well separated from $|\lambda_1 |$. As a result, a few block power iterations allow the $4$th dominant and the smaller eigenvalues of the power matrix to quickly approach $0$. Consequently, the block power method with block size $3$ converges substantially faster than the simple power method as well as the block power method with block size $2$. This is consistent with our analysis in Theorem \ref{theorem1}. 

\subsection{Stopping criteria}
\label{subsec:2.4}

In practice, because the true dominant eigenvector $v_1$ is unknown, the theoretical measure $\text{dist}\left( span \left\lbrace Y_k^{(1)} \right\rbrace,span \left\lbrace v_1 \right\rbrace  \right)$ stated in Theorem \ref{theorem1} cannot be used to determine whether the block power method has reached convergence within certain accuracy or not during block power iterations. Instead, by taking advantage of the fact that the dominant eigenvalue $\lambda_1$ is known to be $1$ in Markov chain transition matrices, we use the error norm $\|P^T v-v\|_2$ as an indicator to estimate how well the computed vector $v$ approximates the actual dominant eigenvector.
To find an optimal vector that minimizes the approximation error $\|P^T v-v\|_2$ over the space spanned by $Q_k$ at the $k$th iteration, we construct a matrix $B_k$ in the following form 
$$B_k=Q_k^T (P^T-I)^T (P^T-I) Q_k.$$
Correspondingly, the decomposition step of the block power method is modified as below:\\
\begin{tabbing}
\hspace*{1.5cm}\= \hspace*{4cm} \= \kill 
	\>$Q_k R_k=W_k	$ \>/* QR Decomposition */\\
	\>$Z_k=(P^T-I)Q_k $ \>/* Projection */\\
	\>$B_k=Z_k^T Z_k$ \\
	\>$V_k \Lambda_k V_k^{-1}=B_k$ \> /* Eigendecomposition */\\
	\>$Y_k=Q_k V_k.$ \\
\end{tabbing}

It is important to note that in this modified decomposition step, the Ritz vector corresponding to the smallest Ritz value in $B_k$ is employed to approximate the left dominant eigenvector of $P$, rather than the one described in Section \ref{subsec:2.2} using the largest Ritz value. Here $B_k$ is a small, square matrix, and thus the Ritz vector with the smallest Ritz value can be easily obtained by performing eigendeomposition on $B_k$. The minimization property of this decomposition step along the block power iterations is proved in Proposition \ref{proposition1}.

\begin{proposition}
\label{proposition1}
Suppose that $Q$ is an $n\times s$ matrix, then
$$v^{*} = \operatorname*{arg\,max}_{v \in ran(Q),\|v\|_2=1 } \|P^T v-v\|_2 $$
and $v^{*}=Q\widetilde{v}_k$ is the minimizer, where $\widetilde{v}_k$ is the eigenvector corresponding to the smallest eigenvalue of matrix $Q^T (P^T-I)^T (P^T-I)Q$ and $I$ is an identity matrix. 
\end{proposition}
\begin{proof}
As $v \in ran(Q)$ and $\|v\|_2=1$, there exists a nonzero vector $y\in R^k$ such that $v=Qy$ and $y^T y=1$.
Then,
\begin{eqnarray*}
\|P^T v-v\|_2^2 &=& (P^T v-v)^T (P^T v-v) \\
&=& v^T (P^T-I)^T (P^T-I)v  \\
&=& y^T Q^T (P^T-I)^T (P^T-I)Qy.
\end{eqnarray*}
Clearly, $Q^T (P^T-I)^T (P^T-I)Q$ is an $s\times s$ symmetric positive--semidefinite matrix. Let $\widetilde{v}_k$ denote its eigenvector corresponding to the smallest eigenvalue. Hence, according to the Courant--Fischer Theorem \cite{golub2012matrix},  $v^{*}=Q\widetilde{v}_k$ is the minimizer of $\|P^T v-v\|_2^2$ over the column space of $Q$ and $\|v^{*}\|_2=1$. 
\end{proof}

\subsection{A sliding window scheme}
\label{subsec:2.5}

Although the block power method can typically converge in fewer iterations than the simple power method, each block power iteration requires $s$ matrix--vector multiplications. The memory requirement of the block power method is also higher than that of the simple power method, due to the storage of the block matrix. Here, we design a sliding window scheme during the power iterations to construct the block matrix.

The fundamental idea of the sliding window scheme is to take advantage of $t$ intermediate subsequent vectors in power iterations to form the multi--dimensional invariant subspace as follows,\\\\
\hspace*{1.5cm}\textbf{for} $i = 1, 2, \cdots, k,\cdots$ \\
\hspace*{2.0cm}$x_i=P^T x_{i-1}$ \\
\hspace*{1.5cm}\textbf{end for}\\
\hspace*{1.5cm}$W_k=\left[x_{k-t+1}, x_{k-t+2}, \cdots, x_k\right].$\\\\
Correlating the vectors in the sliding window can remove the influence from the eigenvalues whose magnitudes are smaller than $|\lambda_{t}|$ and thus the eigendecomposition step on $W_k$ is able to fast reveal the dominant eigenvector. Therefore, ideally, the sliding window scheme is expected to keep the fast convergence capability as the block power method but avoid multiplying $P$ with a block matrix at each iteration step. 

More generally, the sliding window scheme can also be applied to the block power method to build up bigger blocks without adding additional matrix-vector multiplications, i.e.,\\\\
\hspace*{1.5cm}\textbf{for} $i = 1, 2, \cdots, k,\cdots$ \\
\hspace*{2.0cm}$X_i=P^T X_{i-1}$ \\
\hspace*{1.5cm}\textbf{end for}\\
\hspace*{1.5cm}$W_k=\left[X_{k-t+1},X_{k-t+2},\cdots,X_k\right].$\\\\
where $t\times s$ denotes the size of sliding window $W_k$ and, as before, $X_0$ is an $n \times s$ initial block matrix. In fact, the sliding window matrix represents a truncated Krylov subspace based on $X_0$, i.e.,
\begin{eqnarray*}W_k &=& \left[ X_{k-t+1},X_{k-t+2},\cdots,x_k \right] \\
&=& \left[ (P^T)^{k-t+1} X_0, (P^T)^{k-t+2} X_0, \cdots,(P^T)^{k} X_0 \right].
\end{eqnarray*}
Then, the eigendecomposition step can be carried out accordingly on  $W_k^T P^T W_k$ to obtain the Ritz pairs and the approximate dominate eigenvector can be extracted accordingly: \\
\begin{tabbing}
\hspace*{1.5cm}\= \hspace*{4cm} \= \kill 
	\>$Q_k R_k=W_k$	\>/* QR Decomposition */ \\ 
	\>$B_k=Q_k^T P^T Q_k$ \>/* Projection */ \\
	\>$V_k \Lambda_k V_k^{-1}=B_k$ \>/* Eigendecomposition */\\
	\>$Y_k=Q_k V_k.$\\
\end{tabbing}
When the sliding window scheme is incorporated, the block power method is expected to yield a faster convergence, since a bigger block size is used. However, due to the fact that these intermediate vectors in the truncated Krylov subspace are highly correlated, the convergence rate of the block power method with the sliding window scheme depends on the actual rank of $W_k$. In fact, the convergence speed of block power method with sliding window size $t$ is governed by a factor lying between $|\lambda_{t\times s+1} |^k$ and $|\lambda_{s+1} |^k$. In the best case, the block power method with sliding window will have convergence rate related to $\lambda_{t\times s+1}$. However, if the rank of $W_k$ is $s$, the performance of the block power method with sliding window is reduced to the block power method with block size $s$. Nevertheless, if this worse case actually occurs, the block power iteration should have already converged. The sliding window scheme demands additional storage for the subsequent vectors and higher computational cost for decomposing larger block matrices, but without additional rounds of passing over $P$ or additional matrix--vector multiplications. As a result, in practice, the sliding window scheme can usually enhance the computational efficiency of the power method and the block power method, which we will show by examples in the next section.

\section{Numerical results}
\label{sec:3}

In this section, we present numerical results on the calculation of the stationary distribution of a Markov chain model of a stochastic lumincal calcium release site \cite{hys2016cal}. See \cite{hys2016cal} for details of the model. The calcium release site model consists of $N$ calcium--regulated calcium ion channels, in which each channel is represented by a four--state model (i.e., closed, inactivated, closed--inactivated, and open states).  The Markov chain representing the complete calcium release site is composed of all possible calcium channel state configurations and possible values for the discrete number of calcium ions present at the release site.  We have previously used this model to show that fluctuations in calcium ion concentration can alter the stationary properties of calcium release site channel opening and closing \cite{hys2016cal}. Depending on the number of calcium ion channels and model parameters, the corresponding Markov chain transition matrices vary in sizes and relative eigenvalue magnitudes, which provides nice scenarios with different patterns to analyze the convergence of the simple power method, the block power method, and the sliding window scheme. 

\subsection{Calcium release site with two calcium channels}
\label{subsec:3.1}

We first apply the block power method to compute the dominant eigenvector of a $109,690 \times 109,690$ Markov transition matrix (Ca--$N2$) derived from a model with two calcium channels. FIG. \ref{fig:fig3} compares the convergence of the block power method with block sizes ranging from $1$ to $128$. The desired accuracy is set to $10^{-10}$. Eigendecomposition on the block matrix is carried out at every $10,000$ iteration steps to extract the approximate dominant eigenvector and to estimate convergence. One can find that a larger block size can lead to fewer iteration steps to reach the desired accuracy, which agrees with our theoretical analysis on the block power method. In particular, when the block size is $1$, the block power method is equivalent to the simple power method, which takes $50,910,000$ iteration steps to reach the desired accuracy. When a block size of $16$ is used, the total number of iterations of the block power method is reduced to $250,000$, which is approximately $1/200$ of that of the simple power method. The block power method with block size $128$ further brings the total number of iteration steps down to $100,000$. 

\begin{figure}[!htbp]
\centering
\includegraphics[scale=0.25]{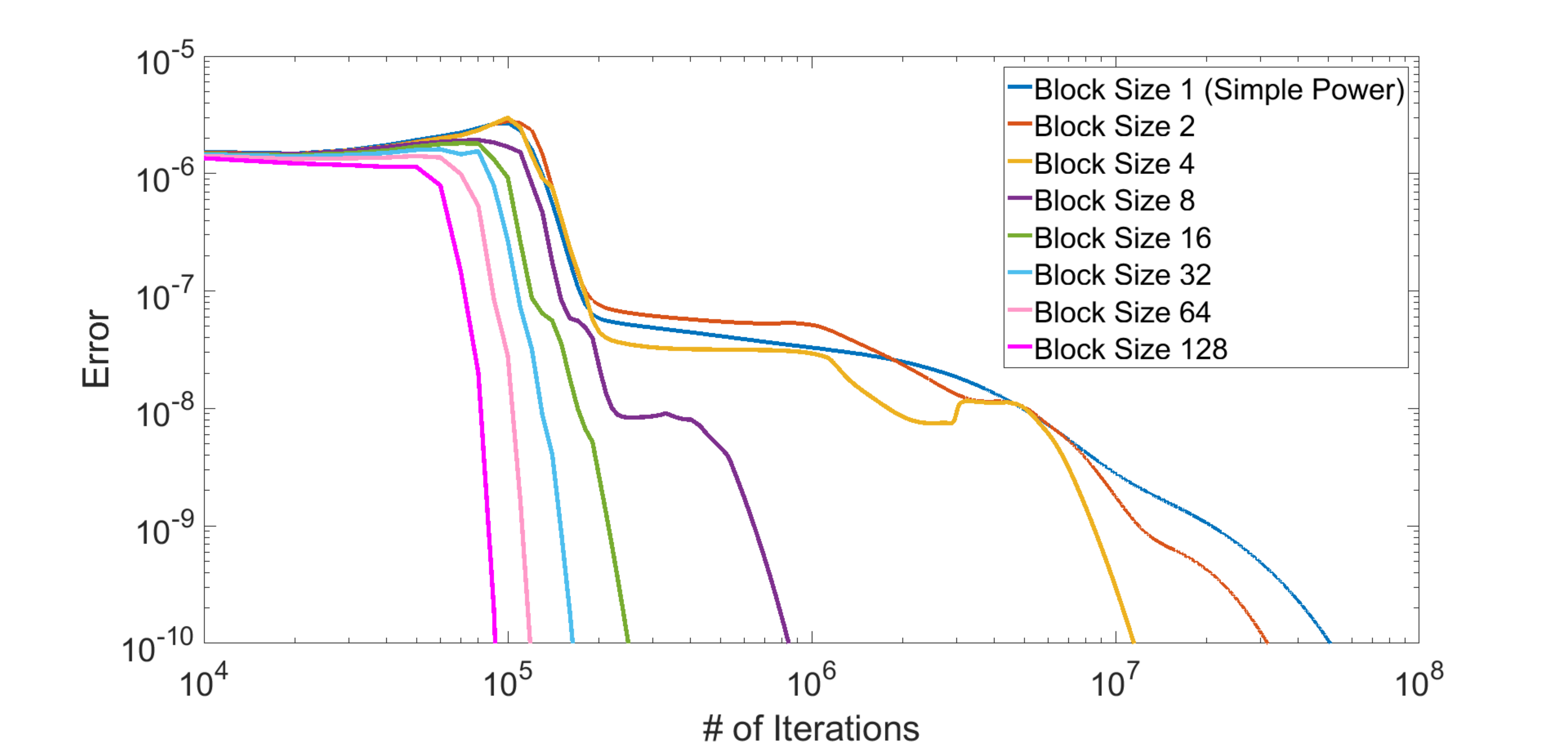}
\caption{Convergence of block power method using various block sizes ($1$--$128$) using the transition matrix for the Ca--$N2$ model}
\label{fig:fig3}       
\end{figure}
	
FIG. \ref{fig:fig4} compares the total numbers of matrix--vector multiplications required to reach convergence for the block power method with different block sizes for the Ca--$N$2 model. We find that the block power method using a block size $16$ is optimal, which requires $92.14\%$ fewer matrix--vector multiplications than the simple power method to obtain a solution with desired accuracy. It is important to note that the number of matrix--vector multiplications used to reach convergence is a good estimate of the performance of the block power method. This is due to the fact that the eigendecomposition step to extract the approximated eigenvector and to estimate convergence is only carried out every $10,000$ iteration steps and thus the computational cost spent on eigendecomposition operations is marginal compared to the vast majority of the computational cost of matrix--vector multiplications in the block power method. 

\begin{figure}[!htbp]
\centering
\includegraphics[scale=0.25]{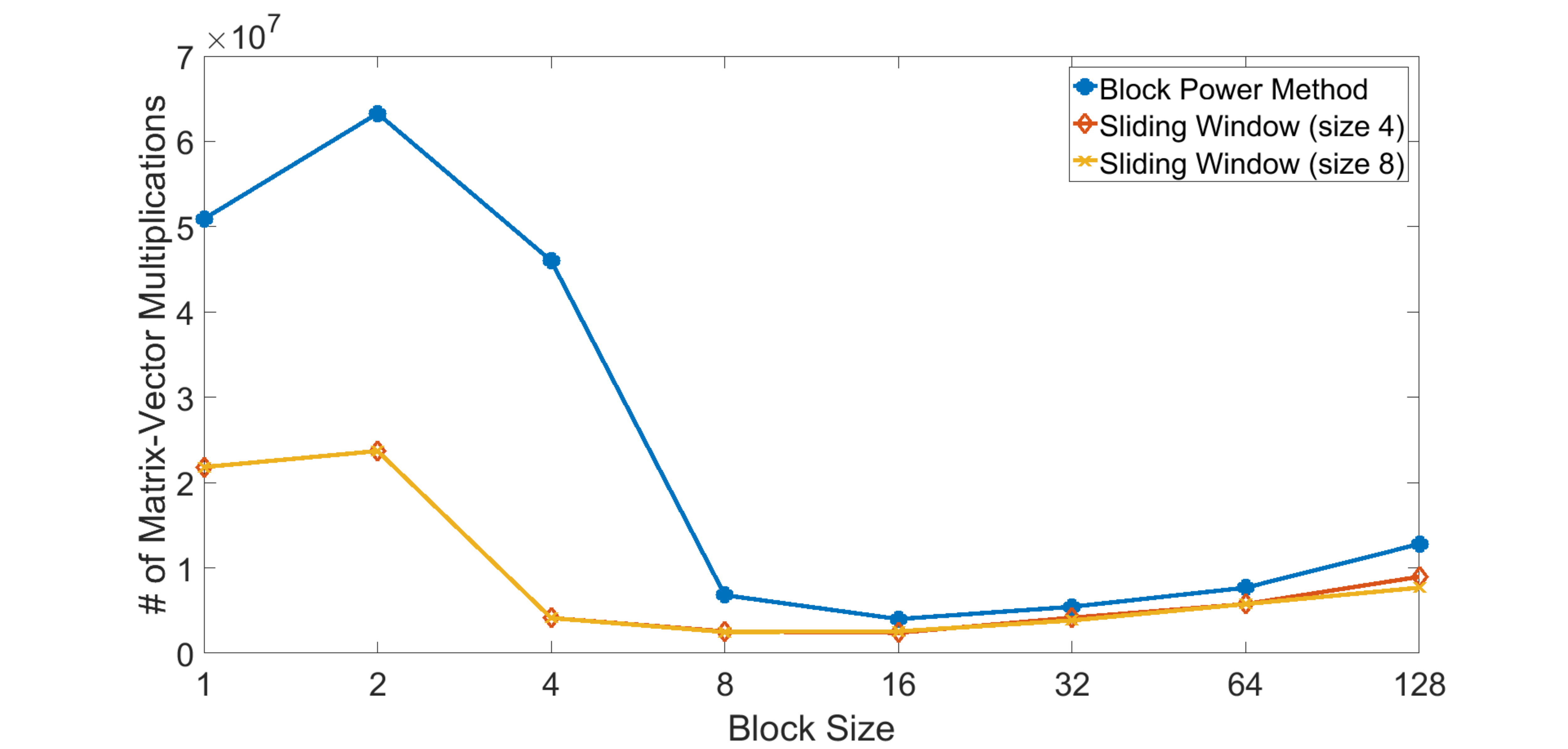}
\caption{The total number of matrix--vector multiplications is shown as a function of the block sizes using the block power method and the block power method with sliding window (sizes $4$ and $8$) for the transition matrix in the Ca--$N2$ model}
\label{fig:fig4}       
\end{figure}

FIG. \ref{fig:fig4} also demonstrates the effectiveness of the sliding window scheme. We find that a sliding window of size $4$ used with the simple power method (i.e., block size $1$) reduces the total number of matrix--vector multiplications required to reach convergence by $57.14\%$. When combining a sliding window with the block power method, the numbers of matrix--vector multiplications required to converge in the desired accuracy are also reduced for different block sizes. However, a larger sliding window does not significantly reduce the total numbers of matrix--vector multiplications any more, indicating that enlarging the sliding window does not provide further useful information to reveal the dominant eigenvector.

\subsection{Calcium release site with $40$ calcium channels}
\label{subsec:3.2}
Next, we used the block power method to calculate the stationary distribution of a $1,246,441 \times 1,246,441$ Markov transition matrix obtained from modeling a calcium release site with $40$ calcium channels (Ca--$N40$). Similar to the Ca--$N2$ model, increasing the block size reduces the total number of iteration steps to reach convergence (FIG. \ref{fig:fig5}). However, when assessed by the total number of matrix--vector multiplications, as shown in FIG. \ref{fig:fig6}, using a block size larger than $1$ does not actually reduce the overall number of matrix--vector multiplications. Our analysis shows that there are significantly more possible ion channel configurations in the Ca--$N40$ model, compared with Ca--$N2$, and thus many more distinct and meaningful calcium release site configurations, resulting in many eigenvalues in the Ca--$N40$ transition matrix with magnitudes are close to $1.0$. As a result, the important properties of the matrix cannot be captured by a small block. However, using the sliding window approach can reduce the total number of matrix--vector multiplications by nearly a half in the block power method with different block sizes. Similar to the Ca--$N2$ model, a small sliding window size ranging from $4$ to $8$ is effective but a larger one does not further improve performance significantly.

\begin{figure}[!htbp]
\centering
\includegraphics[scale=0.25]{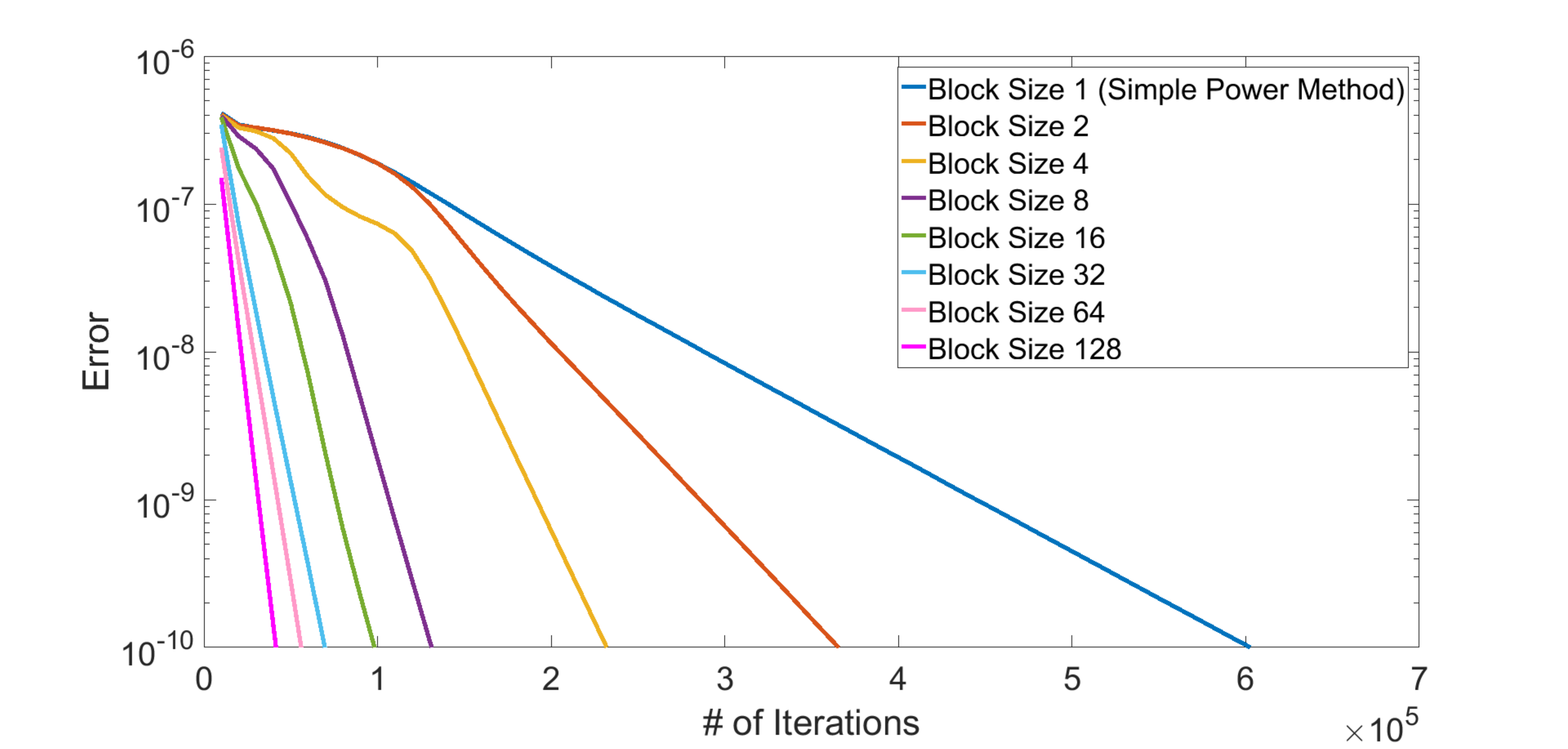}
\caption{Convergence of block power method using various block sizes ($1$--$128$) using the transition matrix for the Ca--$N40$ model}
\label{fig:fig5}       
\end{figure}

\begin{figure}[!htbp]
\centering
\includegraphics[scale=0.476]{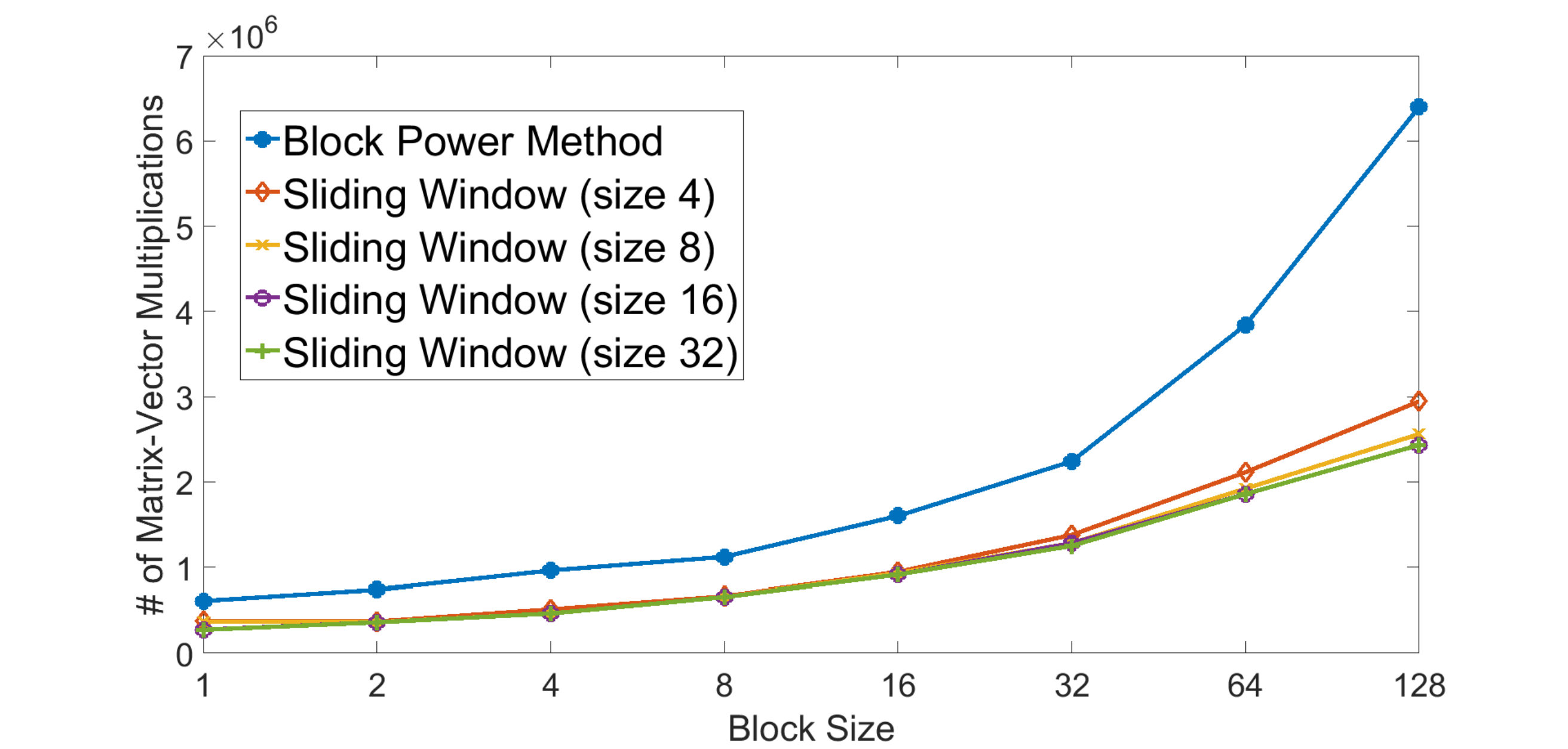}
\caption{The total number of matrix--vector multiplications is shown as a function of block sizes using block power method and block power method with sliding window (sizes $4$, $8$, $16$, and $32$) for the transition matrix in the Ca--$N40$ model}
\label{fig:fig6}       
\end{figure}

\section{Conclusions}
\label{sec:4}

In this paper, we revisit the block power method for approximating the dominant left eigenvector of a Markov transition matrix $P$. When the block matrix consists of $s$ linearly independent vectors, the convergence of the block power method relies on the magnitude of $\lambda_{s+1}$ in $P$. This is a generalization of the fundamental Markov chain theory, in which the power method with a single vector is used and the convergence rate is determined by $|\lambda_2 |$. The block power method can be effectively used to reduce the total number of passes over the transition matrix to reach convergence. We also propose a sliding window scheme to build up blocks from subsequent vectors in the previous iterations. Our computational results show that incorporating a small sliding window with the power method or the block power method can effectively reduce the overall computational cost of finding the Markov chain stationary distribution vector.

\section*{Acknowledgments}
Y.~Li acknowledges support from National Science Foundation through grant number CCF--1066471. H.~Ji acknowledges support from Old Dominion University Modeling and Simulation Fellowship.
\bibliographystyle{siamplain}
\bibliography{mybibfile}
\end{document}